\documentclass[11pt,a4paper]{article}
\usepackage[utf8]{inputenc}
\usepackage[a4paper, total={6.6in, 9.8in}]{geometry}

\usepackage[T1]{fontenc}
\usepackage[variablett]{lmodern}
\usepackage{amsmath}
\usepackage{amsfonts}
\usepackage{amssymb}
\usepackage{float}
\usepackage{wrapfig}
\usepackage{enumitem}
\usepackage{wasysym}

\makeatletter
\newcommand{\biggg}{\bBigg@\thr@@}
\newcommand{\Biggg}{\bBigg@{3.5}}
\newcommand{\vast}{\bBigg@{4}}
\newcommand{\Vast}{\bBigg@{5}}
\makeatother

\usepackage{graphicx}
\usepackage{quiver}
\usepackage{breakurl}

\usepackage[bottom]{footmisc}
\usepackage[backend=bibtex,style=alphabetic]{biblatex}
\bibliography{countcomp}

\usepackage{caption}
\usepackage[hidelinks]{hyperref}
\definecolor{mycolor}{HTML}{750000}
\hypersetup{
    colorlinks = true,
	breaklinks = true,
    allcolors = mycolor,
    bookmarksdepth = 2,
}
\PassOptionsToPackage{unicode}{hyperref}
\PassOptionsToPackage{naturalnames}{hyperref}
\usepackage[nameinlink]{cleveref}
\Crefname{subsection}{Subsection}{Subsections}
\Crefname{question}{Question}{Questions}
\Crefname{subsubsection}{Paragraph}{Paragraphs}

\newlist{propenum}{enumerate}{1}
\setlist[propenum]{label=(\roman*), ref=\theproposition~(\roman*)}
\crefalias{propenumi}{proposition} 

\newlist{thmenum}{enumerate}{1}
\setlist[thmenum]{label=(\roman*), ref=\thetheorem~(\roman*)}
\crefalias{thmenumi}{theorem} 

\newlist{corenum}{enumerate}{1}
\setlist[corenum]{label=(\roman*), ref=\thecorollary~(\roman*)}
\crefalias{corenumi}{corollary} 

\usepackage{comment}

\newcommand{\N}{\mathbb{N}}
\newcommand{\Z}{\mathbb{Z}}
\newcommand{\Zpos}{\Z_{\geq 0}}
\newcommand{\Zsp}{\Z_{>0}}
\newcommand{\Q}{\mathbb{Q}}
\newcommand{\C}{\mathbb{C}}
\newcommand{\R}{\mathbb{R}}
\renewcommand{\AC}{\mathbb{A}^1(\C)}

\renewcommand{\P}{\mathbb{P}}

\newcommand{\PC}{\P^1(\C)}

\newcommand{\Oo}[1]{O\!\left(#1\right)}

\newcommand{\Sym}{\mathfrak{S}}

\newcommand{\card}[1]{ \left | #1 \right | }
\newcommand{\gen}[1]{ \left \langle #1 \right \rangle }
\newcommand{\ord}{\mathrm{ord}}

\newcommand{\Hur}{\mathrm{Hur}} 
\newcommand{\CHur}{\mathrm{CHur}} 
\newcommand{\Conf}{\mathrm{Conf}} 
\newcommand{\B}{\mathrm{B}} 

\newcommand{\Comp}{\mathrm{Comp}} 

\newcommand{\Sub}{\mathrm{Sub}}
\renewcommand{\HF}{\mathrm{HF}}

\newcommand{\gbar}{\underline{g}}
\newcommand{\tbar}{\underline{t}}
\newcommand{\verti}{\, \middle \vert \,}

\newcommand{\Nn}{\mathcal{N}}

\renewcommand{\bar}{\overline}
\renewcommand{\hat}{\widehat}
\renewcommand{\tilde}{\widetilde}

\newcommand{\ab}{\textnormal{ab}}

\renewcommand{\binom}[2]{\left(\begin{matrix}#1\\#2\end{matrix}\right)}

\newcommand{\customref}[2]{\hyperref[#2]{#1}}

\newcommand\iref[2]{\customref{\Cref*{#1}~\ref*{#2}}{#1}}

\usepackage{amsthm}
\usepackage{thmtools}

\newcounter{mycounter}[section]

\theoremstyle{plain}
\newtheorem{theorem}[mycounter]{Theorem}
\newtheorem{corollary}[mycounter]{Corollary}
\newtheorem{proposition}[mycounter]{Proposition}
\newtheorem{lemma}[mycounter]{Lemma}

\theoremstyle{remark}
\newtheorem{remark}[mycounter]{Remark}

\theoremstyle{definition}
\newtheorem{definition}[mycounter]{Definition}

\counterwithin{equation}{section}

\usepackage{titletoc}
\titlecontents{section}
[0pt]                                               
{}%
{\contentsmargin{0pt}                               
    \large\thecontentslabel.\enspace
    }
{\contentsmargin{0pt}\large}                        
{\titlerule*[.5pc]{ }\contentspage}                 
[]                                                  

\usepackage{titlesec}
\titleformat{\section}[block]{\normalfont\centering\scshape\large}{\thesection.}{1em}{}
\titleformat{\subsection}[block]{\normalfont\large}{\thesubsection.}{1em}{\bf}
\titleformat{\subsubsection}[runin]{\normalfont}{\bf\thesubsubsection.}{0.3em}{\bf}

\makeatletter
\patchcmd{\@maketitle}{\LARGE}{\huge}{\typeout{OK 1}}{\typeout{Failed 1}}
\patchcmd{\@maketitle}{\large \lineskip}{\Large \lineskip}{\typeout{OK 2}}{\typeout{Failed 2}}
\makeatother

\title{
  Counting Components of Hurwitz Spaces
}

\author{Béranger Seguin\footnote{Universität Paderborn, Fakultät EIM, Institut für Mathematik, Warburger Str. 100, 33098 Paderborn, Germany. Email: \texttt{bseguin@math.upb.de}.}}
\date{}

\renewenvironment{abstract}{%
\par\noindent\rule{\textwidth}{1pt}
\par\noindent\textsc{Abstract.}}
{\par\noindent\rule{\textwidth}{1pt}}

\begin{document}

\maketitle{}

\begin{abstract}
	For a finite group $G$, we obtain asymptotics for the number of connected components of Hurwitz spaces of marked $G$-covers (of both the affine and projective lines) whose monodromy classes are constrained in a certain way, when the number of branch points grows to infinity.
	More precisely, we compute both the degree and (in many cases) the coefficient of the leading monomial in the count of components of marked $G$-covers whose monodromy group is a given subgroup of $G$.
	By the work of Ellenberg, Tran, Venkatesh and Westerland, these asymptotics are related to the distribution of field extensions of~$\mathbb{F}_q(T)$ with Galois group~$G$.

	\medskip

	\par\noindent
	\textbf{MSC 2020:} 14H30 $\cdot$ 14J10 $\cdot$ 14D22 $\cdot$ 12F12
\end{abstract}

{
  \hypersetup{linkcolor=black}
  \tableofcontents{}
}
\par\noindent\rule{\textwidth}{1pt}

\section{Introduction}

\subsection{Motivation and previous work}

Let $G$ be a finite group and $K$ be a field.
A central question in number theory is to describe the asymptotic count of extensions of $K$ whose Galois closures have Galois group isomorphic to~$G$, as an upper bound on their discriminant grows.
The most influential conjecture in this area was introduced by Malle in \cite{malle02}.
When $K=F(T)$ is a rational function field over a field~$F$, $G$-extensions of $K$ correspond to geometric objects, namely connected branched Galois covers of the projective line $\P^1_F$ whose automorphism group is isomorphic to $G$.
In turn, when $F$ has characteristic coprime to $|G|$, these covers are related
to the $F$-points of \emph{Hurwitz schemes}, which are moduli spaces parametrizing $G$-covers of $\P^1$ whose branch divisor has a fixed degree.
These ideas are detailed in \cite{Fried77,Fried91}.

In \cite{EVW}, Ellenberg, Venkatesh and Westerland set out to count the $\mathbb{F}_q$-points of Hurwitz schemes in order to establish a weak form of the Cohen--Lenstra heuristics over $\mathbb F_q(T)$ for $q$ large \cite[Theorem~1.2]{EVW}.
We describe their strategy shortly.
Using the Grothendieck--Lefschetz trace formula and Deligne's bounds on eigenvalues of Frobenius, they reduce the counting problem to a question about the algebraic topology of Hurwitz spaces \cite[Theorem~8.8]{EVW}.
Using a variant of the Koszul complex, they show that the higher homology of these spaces is controlled by the number of connected components \cite[Theorem~6.1]{EVW}.
They then conclude using a result of homological stability \cite[Lemma~3.5]{EVW}.
This strategy has been extended with success in \cite{ETW} to prove the upper bound in the weak Malle conjecture for $G$-extensions of $\mathbb F_q(T)$ (where~$q$ is coprime with~$|G|$).

\subsection{Approach of this work}

The work of Ellenberg, Tran, Venkatesh and Westerland has made clear that the asymptotic distribution of extensions of $\mathbb F_q(T)$ is related to the growth of Betti numbers of Hurwitz spaces.
In this article, we examine closely the asymptotics of their number of connected components, i.e., the ``zeroth homology''.
The main differences between the approach of \cite{EVW} and ours are the following:
\begin{itemize}
	\item
		Let $c$ be a conjugacy class of $G$ which generates $G$.
		The results of \cite{EVW} depend on the \emph{non-splitting property}: the pair $(G, c)$ is \emph{non-splitting} if for every subgroup $H \subseteq G$ intersecting~$c$, the set $c \cap H$ is a conjugacy class of $H$.
		In the situation of \cite{EVW}, this hypothesis is satisfied.
		We are mostly interested in situations where this assumption does not hold.

		We define a numerical invariant measuring the defect of this hypothesis, the \emph{splitting number} (\Cref{defn:splitting-number}), which vanishes exactly when the non-splitting property holds.
		Our main result, \Cref{thm:asymp-comp-main}, shows the importance of that invariant in the growth of the number of connected components of Hurwitz spaces.
		This generalizes some of the stabilization results of \cite{EVW}.
		We also systematically consider the case of multiple conjugacy classes, as was done in \cite{Tietz}.

		Moreover, we describe the leading coefficient of the corresponding counting function.
	\item
		In \cite{EVW}, the focus is on covers of the affine line, whereas we consider marked branched $G$-covers of both the affine and projective lines.
		The counting is subtler in the projective case.
\end{itemize}

Our results should also be compared to \cite[Theorem~10.4]{lwzb}, where estimates are given for the number of components which are defined over $\mathbb{F}_q$.
This corresponds to the case where $D$ is a singleton in the notation of our \Cref{thm:asymp-comp-main}, but with an additional number-theoretic constraint.
In that situation, the fact that components have to be fixed under the Frobenius map constrains the possibilities for the monodromy conjugacy classes and leads to a smaller exponent: in the ``splitting number'', the number of conjugacy classes that a given conjugacy class splits into is replaced by the number of orbits of such conjugacy classes under the $q$-th powering map.
More generally, it would be interesting to obtain ``number-theoretic versions'' of our \Cref{thm:asymp-comp-main}, counting only components which are defined over various non-algebraically closed fields (e.g., over number fields, cf. \cite{fielddef}).

\paragraph{Disclaimer.}
This article is based on the fourth chapter of \cite{SegThese}.
A first version of these results was made public in the preprint \cite{geomringcomp-old}, where the splitting number (cf.~\Cref{defn:splitting-number}) was shown to be the main exponent in the number of connected components of Hurwitz spaces (cf.~\Cref{sn:exponent}).
This number has since been reused in \cite{bianchi23} (which also deals with higher homology, see also \cite{quandles}) and in \cite{ETW}, where some of our estimates were rediscovered.

\subsection{Presentation of the results}
\label{subsn:presentation}

Fix a nontrivial finite group $G$, a set $D$ of nontrivial conjugacy classes of~$G$ which together generate~$G$, and a map $\xi : D \to \Zsp$, whose values are multiplicities associated to each conjugacy class $\gamma \in D$.
We let $\card{\xi} = \sum_{\gamma \in D} \xi(\gamma)$ and $c = \bigsqcup_{\gamma \in D} \gamma$.

Let $X$ be either the complex affine line $\AC$ or the complex projective line $\PC$.
We denote by $\Hur^*_X(G, D, n \xi)$ the Hurwitz space of marked $G$-covers of $X$ (not assumed to be connected) branched at $n \card{\xi}$ points such that, for each conjugacy class $\gamma \in D$, exactly $n\xi(\gamma)$ branch points have their local monodromy elements in $\gamma$ (see \Cref{ssn:hurwitz} for a more careful description).
When restricting to connected marked $G$-covers, the corresponding Hurwitz space is denoted by $\CHur^*_X(G, D, n \xi)$.

Our main result, \Cref{thm:asymp-comp-main}, describes the count of connected components of $\Hur^*_X(G, D, n \xi)$ asymptotically as $n\to\infty$, i.e., when the number of branch points grows to infinity.

\subsubsection{Terminology for subgroups.}
\label{subsubsn:termsubgp}

We first introduce the following terminology:
\begin{definition}
	\label{defn:dgen}
	A subgroup $H$ of $G$ is \emph{$D$-generated} if the subsets $\gamma \cap H$ for $\gamma \in D$ are all nonempty and collectively generate $H$.
	We let $\Sub_{G, D}$ be the set of subgroups of $G$ which are either trivial or $D$-generated.
\end{definition}

Let $H$ be a $D$-generated subgroup of $G$.
Then:
\begin{itemize}
	\item
		We let
		$
			D_H
			:=
			\left\{
				\gamma \cap H \verti \gamma \in D
			\right\}.
		$
		We let $\xi_H : D_H \to \Zsp$ be the map mapping an element $\gamma \cap H \in D_H$ to the integer $\xi(\gamma) \in \Zsp$, where $\gamma$ is the unique class in $D$ containing $\gamma \cap H$.
	\item
		We define the set $c_H = c \cap H$, which is a union of conjugacy classes of $H$ (recall that $c = \bigsqcup_{\gamma \in D} \gamma$).
		We denote by $D_H^*$ the set of all conjugacy classes of $H$ which are contained in the set $c_H$.
	\item
		We denote by $\tau_H : D_H^* \rightarrow D$ the surjection which maps a conjugacy class $\gamma \in D^*_H$ of $H$ to the conjugacy class $\tau_H(\gamma) \in D$ of $G$ containing $\gamma$.
\end{itemize}

\begin{definition}
	\label{defn:splitting-number}
	The splitting number $\Omega(D_H)$ of $H$ is defined as $\card{D_H^*} - \card{D_H}$.
\end{definition}

Note that the splitting number of $H$ is also given by $\Omega(D_H) = \sum_{\gamma \in D} \left( \card{\tau_H^{-1}(\gamma)} - 1 \right)$.
It is a  nonnegative integer, and it vanishes if and only if $D_H$ consists of conjugacy classes of $H$.
This integer is a quantitative version of the \emph{non-splitting property} introduced in \cite[Definition~3.1]{EVW}: the pair $(G, D)$ is \emph{non-splitting} when $\Omega(D_H) = 0$ for all $D$-generated subgroups~$H$ of~$G$.

\medskip

Separating covers according to their monodromy group, we obtain a homeomorphism:
\begin{equation}
	\label{eqn:hur-sqcup}
	\Hur^*_X(G, D, n \xi)
	=
	\bigsqcup_{H \in \Sub_{G, D}}
		\CHur^*_X(H, D_H, n \xi_H)
\end{equation}
relating the connected components of $\Hur^*_X(G, D, n \xi)$ to those of the Hurwitz spaces of connected $H$-covers, for all $D$-generated subgroups $H$ of~$G$.
Therefore, it suffices to estimate the number of connected components of $\CHur^*(H, D_H, n \xi_H)$ for $H \in \Sub_{G, D}$ to solve the original counting problem.
We denote by $\card{\pi_0 \CHur^*_X(H, D_H, n \xi_H)}$ the number of connected components of $\CHur^*_X(H, D_H, n \xi_H)$.

\subsubsection{Relation with Hilbert functions.}
Assuming that $X=\PC$, we now explain why the number of components grows quasi-polynomially.
In \Cref{subsn:mon-comp}, the set $\Comp_X(D, \xi) := \bigsqcup_{n \geq 0} \pi_0 \Hur^*_X(G, D, n \xi)$ is equipped with the structure of a graded monoid.
Let $R$ be the $\Q$-algebra of that monoid, which is a variant of the ``ring of components'' of~\cite{EVW}.
For a subgroup $H \subseteq G$, let~$R^H$ be the subalgebra of~$R$ generated by components of $G$-covers with monodromy group contained in~$H$.
\Cref{propitem:product-one-commute,prop:fingen} 
imply that $R^H$ is a finitely generated commutative graded $\Q$-algebra (whose generators need not be of degree~$1$).
The dimension of the space of homogeneous elements of degree~$n$ in~$R^H$ equals the number of connected components of $\Hur^*_{\PC}(H, D_H, n \xi_H)$, which decomposes as in \Cref{eqn:hur-sqcup}.
In particular, the Hilbert function $\HF_{R^H}$ of $R^H$ satisfies:
\[
	\HF_{R^H}(n)
	=
	\sum_{H' \subseteq H}
		\card{\pi_0 \CHur^*_X\big(
			H', D_{H'}, n \xi_{H'}
		\big)}.
\]
From this, one deduces an expression of $\card{\pi_0 \CHur^*_X(H, D_H, n \xi_H)}$ as a weighted sum of Hilbert functions of finitely many finitely generated commutative graded $\Q$-algebras, using a generalized Möbius inversion formula on the lattice of subgroups of~$G$.
This observation, coupled with a variant of the Hilbert--Serre theorem for non-standard graded algebras (cf.~\cite[Theorem~4.4.3]{cohmac}, \cite[Lemmas~5.4.4, 5.4.5 and 5.4.6]{SegThese}%
), has the following consequence for the counting function:
\begin{proposition}
	\label{prop:hilbfun-oscpol}
	There is an integer $W \geq 1$ and integer-valued polynomials $Q_0, \ldots, Q_{W-1} \in \Q[X]$, with $\deg Q_0 \geq \deg Q_i$ for all $i \in \{1,\ldots,W-1\}$, such that, for all $n \in \N$ large enough, we have $\card{\pi_0 \CHur^*_X(H, D_H, n \xi_H)} = Q_{n \bmod W}(n)$.
\end{proposition}

In other words, the counting function $\card{\pi_0 \CHur^*_X(H, D_H, n \xi_H)}$ alternates periodically (for large~$n$) between finitely many polynomials $Q_0,\ldots,Q_{W-1}$ of degree at most $\deg Q_0$.
This has the consequence that $\card{\pi_0 \CHur^*_X(H, D_H, n \xi_H)}$ essentially grows polynomially.
This can be formalized as follows: an average order%
\footnote{
	A function $g : \N \to \R$ is an \emph{average order} of a function $f : \N \to \R$ if
	$
		\sum_{k=0}^n
		f(k)
		\underset
			{n \to \infty}
			{\sim}
		\sum_{k=0}^n
		g(k)
	$.
}
of $\card{\pi_0 \CHur^*_X(H, D_H, n \xi_H)}$ is given by a monomial of degree $\deg Q_0$, whose coefficient is the average of the coefficients in front of $n^{\deg Q_0}$ in the polynomials $Q_0, \ldots, Q_{W-1}$.
Our main goal is to describe that monomial, and thus the ``average growth'' of $\card{\pi_0 \CHur^*_X(H, D_H, n \xi_H)}$.

\subsubsection{Main theorem.}

We fix a nontrivial $D$-generated subgroup $H$ of $G$.
The commutator subgroup of $H$ is denoted by $[H, H]$, and the abelianization of $H$ by $H^{\ab} := H/[H,H]$.
If $\gamma$ is a conjugacy class of $H$, we denote by $\tilde\gamma$ the element of $H^{\ab}$ obtained as the projection of any element of $\gamma$.
If $\psi$ is a map $D^*_H \to \Z$, we define the following element of $H^{\ab}$ (see also \Cref{defn:tilde-pi}):
\[
	\tilde\pi(\psi)
	=
	\prod_{\gamma \in D^*_H}
		\tilde\gamma^{\psi(\gamma)}.
\]
In particular, if $\Omega(D_H) = 0$, then $D^*_H = D_H$ and we can define an element $\tilde\pi(\xi_H) \in H^{\ab}$.

The group $H_2(H, c_H)$ is defined (following the notation of \cite{wood}) as the quotient of the second homology group $H_2(H, \Z)$ by the subgroup generated by commutators between lifts of commuting elements of $c_H$ in a Schur extension --- we define this group more carefully in \Cref{subsubsn:desc-ugc}.

\bigskip

We finally state our main result:

\begin{theorem}
	\label{thm:asymp-comp-main}
	~
	\begin{thmenum}
		\item
			\label{thmitem:asymp-comp-main-1}
			If $X = \AC$, then:
			\[
				\card{\pi_0 \CHur^*_{\AC}(H, D_H, n \xi_H)}
				\underset{n\to\infty}{=}
				C
				\cdot
				n^{\Omega(D_H)}
				+
				\Oo{n^{\Omega(D_H)-1}}.
			\]
			where
			\[
				C = 
				\card{[H, H]}
				\cdot
				\card{H_2(H, c_H)}
				\cdot
				\prod_{\gamma \in D}
					\frac{
						\xi(\gamma)^{
							\card{
								\tau_H^{-1}(\gamma)
							} - 1
						}
					}{
						\left(
							\card{
								\tau_H^{-1}(\gamma)
							} - 1
						\right)!
					}.
			\]
			Here, $\tau_H$ is the surjection $D_H^* \to D$ defined in \Cref{subsubsn:termsubgp}, so that $\card{
				\tau_H^{-1}(\gamma)
			}$ is the number of conjugacy classes of $H$ contained in $\gamma$, for each $\gamma \in D$.
			This result is \Cref{thm:asymp-affine}.
		\item
			\label{thmitem:asymp-comp-main-2}
			If $X = \PC$, if $\Omega(D_H)=0$, and if $k$ is the order of $\tilde\pi(\xi_H) \in H^{\ab}$, then:
			\[
				\card{\pi_0 \CHur^*_{\PC}(H, D_H, n \xi_H)}
				=
				\left \lbrace
				\begin{matrix}
					0 & \text{if } n \not\in k\N \\
					\card{H_2(H, c_H)} & \text{otherwise, for $n$ large enough.}
				\end{matrix}
				\right .
			\]
			This is \Cref{prop:hilb-func-non-splitt}.
		\item
			\label{thmitem:asymp-comp-main-3}
			If $X = \PC$, if $D = \{c\}$ is a singleton, and if $\xi(c)=1$, then an average order of the counting function $\card{\pi_0 \CHur^*_{\PC}(H, \{c_H\}, n)}$ is given by:
			\[
				\frac{
					\card{H_2(H, c_H)} n^{\Omega(D_H)}
				}{
					\card{H^{\ab}}\Big(\Omega(D_H)\Big)!
				}.
			\]
			This is \Cref{cor:leading-coeff}.
		\item
			\label{thmitem:asymp-comp-main-4}
			If $X = \PC$, we have the upper bound:
			\[
				\card{\pi_0 \CHur^*_{\PC}(H, D_H, n \xi_H)} = O\!\left(n^{\Omega(D_H)}\right)
			\]
			and the lower bound:
			\[
				n^{\Omega(D)} = O\Big(
					\card{\pi_0 \CHur^*_{\PC}(H, D_H, \exp(G) n \xi_H)}
				\Big)
			\]
			where $\exp(G)$ denotes the exponent of the group $G$.
			This is \Cref{thmitem:comp-asymp-proj}.
	\end{thmenum}
\end{theorem}

If $H$ is a nontrivial subgroup of $G$ with $\Omega(D_H)=0$, then \Cref{thm:asymp-comp-main} implies that the number of components of $\CHur^*_X(H, D_H, n\xi_H)$ is a bounded function of $n$.
This is a form of the homological stability for the $0$-th homology of Hurwitz spaces established in \cite{EVW}.

\medskip

In order to estimate the growth of the number of components of $\Hur^*_X(G, D, n\xi)$ using \Cref{thm:asymp-comp-main}, one should first determine for which subgroups $H \in \Sub_{G, D}$ the maximal value of the splitting number $\Omega(D_H)$ is reached.
The degree of the leading monomial will then be this maximal splitting number.
The leading coefficient (which has an unambiguous definition when $X$ is the affine line, and can be given a meaning when $X$ is the projective line by using average orders) is the sum of the leading coefficients associated to subgroups with maximal splitting numbers --- those leading coefficients, in turn, are given by \Cref{thm:asymp-comp-main} in many cases.

\begin{remark}
	\label{rk:symgp}
	In \cite[Theorem~6.1.2]{SegThese}, we fully treat the classical case where $G$ is the symmetric group $\Sym_d$ and $c$ is the conjugacy class of transpositions.
	Let $d' = \lfloor \frac d2 \rfloor$.
	We show that:
	\[
		\card{
			\pi_0 \Hur^*_{\PC}\big(
				G, \{c\}, n
			\big)
		}
		=
		\left\lbrace
			\begin{array}{ccc}
				0 & \text{if $n$ is odd,} \\
				P_d(\frac n2) & \text{if $n$ is even and large enough,}
			\end{array}
		\right.
	\]
	where $P_d$ is a polynomial of leading monomial
	\begin{align*}
		\frac{d!}{2^{d'}(d')!(d'-1)!} \, n^{d'-1} & & \text{ if $d$ is even,} \\
		\left ( 1 + \frac{d'}{3} \right ) \frac{d!}{2^{d'}(d')!(d'-1)!} \, n^{d'-1} & & \text{ if $d$ is odd.}
	\end{align*}
\end{remark}

\subsubsection{Number-theoretic meaning.}
The growth of the number of connected components of Hurwitz spaces is related to the distribution of field extensions of $\mathbb{F}_q(T)$ with a given Galois group, cf.~\cite{EVW,EVW2,ETW}.
In \cite[Theorem~8.8]{ETW}, the Gelfand--Kirillov dimension of the ring of components, whose value follows%
\footnote{%
	See also the finer result \cite[Theorem~5.1.4]{SegThese} which computes the dimension of each stratum in a specific stratification of the ring of components.
}
from \Cref{thmitem:comp-asymp-aff}, appears in the upper bound for the count of certain field extensions of $\mathbb{F}_q(T)$.
More precisely, the log-exponent $d$ in the upper bounds equals the maximal splitting number of a subgroup of~$G$.

We want to clarify the following point: the fact that this computation involves invariants computed in the subgroups of~$G$, and not only in $G$, is \emph{not} just an artifact of the fact that we are looking at group homomorphisms from an absolute Galois group into $G$ without requiring surjectivity (i.e., including non-connected $G$-covers, corresponding on the algebraic side to étale $\mathbb{F}_q(T)$-algebras which are not fields), which obviously leads to overcounting.
Indeed, there is another reason: if one wants to count all $G$-extensions of~$\mathbb{F}_q(T)$ (and not just regular extensions), it is necessary to take some of the covers which are not geometrically connected into account, as an extension $F|\mathbb{F}_q(T)$ with Galois group~$G$ may have a smaller ``geometrical Galois group''  --- i.e., the compositum $F\,\bar{\mathbb{F}_q}(T)$, as an extension of~$\bar{\mathbb{F}_q}(T)$, may have a proper subgroup of~$G$ as its Galois group.
The approach where one counts only those covers which are geometrically connected certainly ensures that these covers are connected over~$\mathbb{F}_q$ whenever they are defined over $\mathbb{F}_q$, but it also leads to undercounting.


\subsection{Outline}

We outline the contents of this article:
\begin{itemize}
	\item
		In \Cref{sn:presentation-objects}, we briefly present our notational and terminological conventions as well as our main objects, namely $G$-covers, Hurwitz spaces, and monoids of components.
	\item
		In \Cref{sn:exponent}, we prove that the splitting number $\Omega(D_H)$ is the smallest exponent of a monomial which dominates the count of the components of $\Hur^*_X(H, D_H, n\xi_H)$ (\Cref{thm:exponent}).
	\item
		In \Cref{sn:leading-coeff}, we focus on the coefficient of the leading monomial of the counting function.
		In various settings, we compute the leading coefficient using group-theoretic information (\Cref{prop:compon-per-multidisc-affine}, \Cref{prop:hilb-func-non-splitt} and \Cref{cor:leading-coeff}).
\end{itemize}

\subsubsection{Organization of the proof.}
\label{sssn:organization}

In \Cref{sn:exponent,sn:leading-coeff}, we prove the various parts of \Cref{thm:asymp-comp-main}.
\Cref{thm:asymp-comp-main} was stated in terms of a subgroup $H$ of $G$.
However, in \Cref{sn:exponent,sn:leading-coeff}, we do everything ``intrinsically'': we count the components of $\CHur^*_X(G, D, n\xi)$ for a general triple $(G, D, \xi)$ where:
\begin{itemize}
	\item
		$G$ is a finite group (playing the role of $H$ above);
	\item
		$D$ is a set of disjoint non-empty subsets of $G$, each of which is a union of conjugacy classes of~$G$, and which collectively generate $G$;
	\item
		$\xi$ is a map $D \to \Zsp$.
\end{itemize}
In this new setting, we forget about the fact that each subset $\gamma \in D$ comes from a single conjugacy class in a larger group.
Instead, elements of $D$ are not assumed to be conjugacy classes anymore, and the splitting number is recovered intrinsically as $\Omega(D) = \card{D^*} - \card{D}$, where $D^*$ is the set of conjugacy classes of $G$ which are contained in $c := \bigsqcup_{\gamma\in D} \gamma$.
This can be used to deduce \Cref{thm:asymp-comp-main} as stated above by applying the results to the triple $(H, D_H, \xi_H)$.

\subsection{Acknowledgments.}

This work was funded by the French ministry of research through a CDSN grant, and by the Deutsche Forschungsgemeinschaft (DFG, German Research Foundation) --- Project-ID 491392403 --- TRR 358.

I thank my PhD advisors Pierre Dèbes and Ariane Mézard for their support and their precious advice, the reporters of my thesis Jean-Marc Couveignes and Craig Westerland for providing helpful feedback, and the anonymous referee for their thorough proofreading and their insightful comments.

\section{Presentation of the objects}
\label{sn:presentation-objects}

In this section, we introduce our main objects.
\Cref{ssn:conventions,ssn:tuples} present our general notational conventions.
\Cref{ssn:gcovers} is a quick explanation of what we call a ``$G$-cover'', to avoid ambiguities.
In \Cref{ssn:hurwitz}, we briefly present Hurwitz spaces and our notation for them, as well as the combinatorial description of their connected components.
\Cref{subsn:mon-comp} contains a definition of the monoid of components of Hurwitz spaces as well as a few classical results concerning the braid group action on tuples of elements of $G$.
Finally, \Cref{subsn:char-not} contains a short chart of notations.

\subsection{General conventions}
\label{ssn:conventions}

The cardinality of a set $X$ is denoted by $\card{X}$.
The exponent of a group $G$ is denoted by $\exp(G)$.
The order of an element $g$ in a group $G$ is denoted by $\ord(g)$.
If $\gamma$ is a conjugacy class of a group $G$, then~$\ord(\gamma)$ denotes the common order of all elements of $\gamma$.

\subsection{Tuples}
\label{ssn:tuples}

Tuples (ordered lists) are denoted by underlined roman letters.
Let $\gbar = (g_1, \ldots, g_n)$ be a tuple of elements of a group $G$.
Its \emph{size} $\card{\gbar}$ is the integer $n$, its \emph{product} $\pi\gbar$ is $g_1 g_2 \cdots g_n \in G$, and its \emph{group} $\gen{\gbar}$ is the subgroup of $G$ generated by $g_1, \ldots, g_n$.
Let $c$ be a union of conjugacy classes of~$G$, containing all elements $g_1, \ldots, g_n$, and let $D^*$ be the set of all conjugacy classes of $G$ contained in $c$.

\begin{definition}
	\label{defn:multidisc}
	The \emph{$c$-multidiscriminant} $\mu_c$ of $\gbar$ is the map $D^* \to \Z$ mapping a conjugacy class $\gamma \in D^*$ to the number of entries of $\gbar$ belonging to $\gamma$:
	\[
		\mu_c(\gamma)
		=
		\card{
			\Big\{
					i \in \{1,\ldots, n\}
				\,\Big\vert\,
				g_i \in \gamma
			\Big\}
		}.
	\]
\end{definition}

\begin{definition}
	\label{defn:tilde-pi}
	Let $\psi$ be a map $D^* \to \Z$.
	We define $\tilde\pi(\psi) \in G^{\ab}$ as the product over conjugacy classes $\gamma \in D^*$ of the elements $\tilde\gamma^{\psi(\gamma)}$, where $\tilde\gamma$ is the common image in $G^{\ab}$ of all elements of~$\gamma$.
\end{definition}

The group homomorphism $\tilde\pi : \Z^{D^*} \to G^{\ab}$ allows one to compute the image of $\pi\gbar$ in $G^{\ab}$ from the $c$-multidiscriminant only, as this equals $\tilde\pi(\mu_c(\gbar))$.

\subsection{$G$-covers}
\label{ssn:gcovers}

Let $X$ be either the complex affine line $\AC$ or the complex projective line $\PC$, equipped with their usual analytic topology, and fix a basepoint $t_0 \in X$.

\subsubsection{Configurations, braid groups.}
\label{subsubsn:conf-braid}

A \emph{configuration} is an unordered list $\tbar$ of points of $X \setminus \{t_0\}$.
Configurations form a topological space $\Conf_{n, X \setminus \{t_0\}}$, whose fundamental group is the \emph{braid group}~$\B_{n, X \setminus \{t_0\}}$.
If $X = \PC$, this braid group is isomorphic to the Artin braid group~$\B_n$, which admits the following presentation:
\[
	\B_n = 
	\gen{
		\sigma_1, \sigma_2, \ldots, \sigma_{n-1}
		\verti
		\begin{matrix}
			\sigma_i \sigma_j
			& =
			& \sigma_j \sigma_i
			& \text{if } |i-j| > 1
			\\
			\sigma_i \sigma_{i+1} \sigma_i
			& =
			& \sigma_{i+1} \sigma_i \sigma_{i+1}
			& \text{if } i < n-1
		\end{matrix}
	}.
\]

\subsubsection{Topological $G$-covers.}

Let $\tbar \in \Conf_{n, X \setminus \{t_0\}}$ be a configuration.
A \emph{$G$-cover of $X$ branched at~$\tbar$} is a (topological) covering map $p : Y \to X \setminus \tbar$ equipped with a group homomorphism $\alpha$ from~$G$ to the group of automorphisms (deck transformations) of the cover, such that $\alpha$ induces a free transitive action of~$G$ on each fiber of $p$.
A $G$-cover of $X$ is necessarily a Galois cover of degree $\card{G}$.
A \emph{marked} $G$-cover is moreover equipped with a marked point in the fiber above the basepoint $t_0$.

Note that this definition allows trivial ramification at the ``branch points'', even if the $G$-cover can be extended into a $G$-cover with fewer branch points.
Moreover, we do not require that $G$-covers be connected.
In particular, a $G$-cover does not necessarily have $G$ as its automorphism group.

\subsubsection{Monodromy and branch cycle description.}

Let $\tbar \in \Conf_{n, X \setminus \{t_0\}}$ be a configuration.
We choose a \emph{topological bouquet} $(\gamma_1, \ldots, \gamma_n)$, which is a list of generators of $\pi_1(X \setminus \tbar, t_0)$ (generating freely if $X = \AC$, and with the nontrivial relation $\gamma_1 \cdots \gamma_n = 1$ if $X=\PC$) such that $\gamma_i$ is the homotopy class of a loop which rotates once counterclockwise around $t_i$ and does not rotate around other branch points.
This choice allows one to describe marked $G$-covers of~$X$ branched at $\tbar$ combinatorially: isomorphism classes correspond exactly to tuples $\gbar = (g_1, \ldots, g_n) \in G^n$ of elements of $G$, with the additional condition that $g_1 \cdots g_n = 1$ if $X = \PC$.
The tuple $\gbar$ is the \emph{branch cycle description} of the marked $G$-cover, and $g_i$ is the \emph{local monodromy element} at the branch point~$t_i$.
The conjugacy classes of the monodromy elements (the \emph{monodromy classes}) do not depend on the choice of the bouquet.
The marked $G$-cover associated to a tuple $\gbar$ is connected if and only if $\gen{\gbar} = G$.
In general, $\gen{\gbar}$ is a subgroup of $G$, which we call the \emph{monodromy group} of the marked $G$-cover.

\subsubsection{Concatenation.}

Given two marked $G$-covers of $X$, branched at disjoint configurations $\tbar$ and~$\tbar'$ (for which bouquets are chosen), their \emph{concatenation} is the marked $G$-cover branched at $\tbar\sqcup\tbar'$ whose branch cycle description is the concatenation of both branch cycle descriptions:
\[
	(g_1, \ldots, g_n)(g'_1,\ldots,g'_{n'}) = (g_1, \ldots, g_n, g'_1, \ldots, g'_{n'}).
\]
Its number of branch points is the sum of the original numbers of branch points, and its monodromy group is the subgroup of $G$ generated by both monodromy groups.

\subsection{Hurwitz spaces and components.}
\label{ssn:hurwitz}

\subsubsection{Hurwitz spaces.}
\label{subsubsn:hurwitz}

Let $X$ be $\AC$ or $\PC$, and let $t_0 \in X$ be a basepoint.
We denote by $\Hur^*_X(G, n)$ the set of pairs $(\tbar, p)$ where $\tbar \in \Conf_{n, X \setminus \{t_0\}}$ is a configuration and $p$ is an isomorphism class of marked $G$-covers of $(X, t_0)$ branched at $\tbar$.
This set is equipped with a topology making the map $(\tbar, p) \mapsto \tbar$ into a covering map $\Hur^*_X(G, n) \to \Conf_{n, X \setminus \{t_0\}}$.
We denote by $\CHur^*_X(G, n)$ the subspace of $\Hur^*_X(G, n)$ consisting of pairs $(\tbar, p)$ where $p$ is a \emph{connected} $G$-cover.

Let $D$ be a set of disjoint non-empty subsets of $G$, each of which is a union of conjugacy classes of $G$, let $\xi$ be a map $D \to \Zpos$, and let $\card{\xi} = \sum_{\gamma \in D} \xi(\gamma)$.
We denote by $\Hur^*_X(G, D, \xi)$ the subspace of $\Hur^*_X(G,\card{\xi})$ consisting of pairs $(\tbar, p)$ such that, for each $\gamma \in D$, exactly $\xi(\gamma)$ branch points have their local monodromy elements in $\gamma$.
We also let $\CHur^*_X(G, D, \xi) := \CHur^*(G, \card{\xi}) \cap \Hur^*_X(G, D, \xi)$.

\subsubsection{Connected components.}
\label{sssn:conn-comp}

Hurwitz spaces are generally not connected.
One can determine whether two marked $G$-covers branched at a same configuration $\tbar \in \Conf_{n, X \setminus \{t_0\}}$ are in the same connected component by looking at their respective branch cycle descriptions $\gbar, \gbar' \in G^n$.
Indeed, there is an action of the Artin braid group $\B_n$ on $n$-tuples of elements of $G$, induced by the formula:
\[
	\sigma_i . (g_1, \ldots, g_n)
	=
	(g_1, \ldots, g_{i-1}, g_i g_{i+1} g_i^{-1}, g_i, g_{i+2}, \ldots, g_n),
\]
and connected components of $\Hur^*_{X}(G, n)$ are in one-to-one correspondence with orbits of $n$-tuples $\gbar \in G^n$ under this action, with the additional condition that $\pi\gbar = 1$ if $X = \PC$.
Similarly, connected components of $\Hur^*_X(G, D, \xi)$ correspond to $B_{\card{\xi}}$-orbits of tuples $\gbar \in G^{\card{\xi}}$ (with $\pi\gbar = 1$ if $X=\PC$) such that exactly $\xi(\gamma)$ entries of the tuple~$\gbar$ are in  each $\gamma \in D$.
For connected components of $\CHur^*_X(G, n)$ (resp.~of $\CHur^*_X(G, D, \xi)$), one also requires $\gen{\gbar} = G$.

We denote by $\sim$ the equivalence relation on tuples of elements of $G$ induced by the actions of the braid groups, i.e., $\gbar \sim \gbar'$ if and only if $\gbar$ and $\gbar'$ have the same size~$n$ and are in the same $\B_n$-orbit.

\subsubsection{Main invariants of components.}
\label{sssn:main-inv}
The size, product, group and $c$-multidiscriminant of a tuple are invariant under the action of the braid group \cite[Proposition~3.3.8]{SegThese}.
Thus, these are well-defined invariants of components of Hurwitz spaces: we extend the use of the notations $\card{x}$, $\pi(x)$, $\gen{x}$ and $\mu_c(x)$ (cf.~\Cref{ssn:tuples}) when $x$ is a component of a Hurwitz space of marked $G$-covers or a braid group orbit of tuples of elements of $G$.


\subsection{The monoid of components.}
\label{subsn:mon-comp}

The following lemmas concerning the braid group action are classical, cf.~\cite[Lemme 2.8, p.564]{Cau} for \Cref{propitem:rotate-product-one} and \cite[Lemme 2.11, p.567]{Cau} for \Cref{propitem:conjugate-subtuples}.
Proofs are also given in \cite[Proposition~3.3.11]{SegThese}:

\begin{proposition}
	\label{prop:prop-braid-action}
	The equivalence relation $\sim$ satisfies the following properties:
	\begin{propenum}
		\item
			\label{propitem:concatenation-compatible}
			If $\gbar_1 \sim \gbar_2$ and $\gbar'_1 \sim \gbar'_2$, then $\gbar_1 \, \gbar'_1 \sim \gbar_2 \, \gbar'_2$.
		\item
			\label{propitem:swap-tuples}
			If $\gbar_1$, $\gbar_2$ are tuples of elements of $G$, then
			$
				\gbar_1 \, \gbar_2
				\sim
				(\gbar_2)^{\pi \gbar_1} \, \gbar_1
				\sim
				\gbar_2 \, (\gbar_1)^{(\pi \gbar_2)^{-1}}
			$.
		\item
			\label{propitem:product-one-commute}
			If $\gbar_1$, $\gbar_2$ are tuples of elements of $G$ with $\pi\gbar_1=1$, then $\gbar_1 \, \gbar_2 \sim \gbar_2 \, \gbar_1$.
		\item
			\label{propitem:rotate-product-one}
			If $\gbar$ is a tuple of elements of $G$ with $\pi\gbar = 1$, then
			$
				(g_1, \, g_2, \, g_3, \, \ldots, \, g_n)
				\sim
				(g_2, \, g_3, \, \ldots, \, g_n, \, g_1)
			$.
		\item
			\label{propitem:conjugate-product-one}
			If $\gbar$ is a tuple of elements of $G$ with $\pi \gbar = 1$ and $\gamma \in \gen{\gbar}$, then
			$
				\gbar \sim \gbar^{\gamma}
			$.
		\item
			\label{propitem:conjugate-subtuples}
			If $\gbar_1$, $\gbar_2$, $\gbar_3$ are tuples of elements of $G$ with $\pi \gbar_2 = 1$ and $\gamma \in \gen{\gbar_1, \gbar_3}$, then
			$
				\gbar_1 \, \gbar_2 \, \gbar_3
				\sim
				\gbar_1 \, \gbar_2^{\gamma} \, \gbar_3
			$.
	\end{propenum}
\end{proposition}

Let $c$ be a union of conjugacy classes of $G$, and let $D^*$ be the set of conjugacy classes of~$G$ contained in $c$.
\Cref{propitem:concatenation-compatible} implies that the concatenation of tuples of elements of $c$ induces a well-defined concatenation operation on equivalence classes of tuples of elements of $c$.
This lets us define the following graded monoid:

\begin{definition}
	\label{defn:monoid-comp}
	The \emph{monoid of components $\Comp_{\AC}(c)$} is the graded monoid whose elements of degree $n$ are $\B_n$-orbits of $n$-tuples of elements of $c$, and whose multiplication is the operation induced by concatenation of tuples.
\end{definition}

Elements of degree $n$ of $\Comp_{\AC}(G)$ correspond to connected components of $\Hur^*_{\AC}(G, n)$ (i.e.,
$
	\Comp_{\AC}(G)
	\simeq
	\bigsqcup_{n \geq 0}
		\pi_0 \Hur^*_{\AC}(G, n)
$
as graded sets).
Hence, we call elements of $\Comp_{\AC}(c)$ \emph{components}.
If $x,y \in \Comp_{\AC}(c)$, we simply write $xy$ to denote the concatenated component.

The monoid $\Comp_{\AC}(c)$ is generated by the $1$-tuples $(g)$ for $g \in c$.
A generating set of relations is given by the action of the elementary braids: $(g)(h) = (h^g)(g)$ for all $g, h \in c$.
The map $(g) \to g$ induces a well-defined monoid homomorphism $\pi : \Comp_{\AC}(c) \to G$.
The element $\pi(x)$ is the \emph{product} of the component $x \in \Comp_{\AC}(c)$.

\begin{definition}
	\label{defn:monoid-comp-pc}
	The \emph{monoid of components $\Comp_{\PC}(c)$} is the kernel of the homomorphism $\pi$.
	It is the graded submonoid of $\Comp_{\AC}(c)$ whose elements of degree $n$ are $\B_n$-orbits of $n$-tuples $\gbar \in c^n$ satisfying $\pi \gbar = 1$.
\end{definition}

Elements of degree $n$ of $\Comp_{\PC}(G)$ correspond to connected components of the Hurwitz space $\Hur^*_{\PC}(G, n)$ (i.e.,
$
	\Comp_{\PC}(G)
	\simeq
	\bigsqcup_{n \geq 0}
		\pi_0 \Hur^*_{\PC}(G, n)
$
as graded sets).
\Cref{propitem:product-one-commute} implies that $\Comp_{\PC}(c)$ is a central (in particular commutative) submonoid of $\Comp_{\AC}(c)$.
By \Cref{sssn:main-inv}, the $c$-multidiscriminant defines a monoid homomorphism $\mu_c : \Comp_X(c) \to \Z^{D^*}$ (where $X$ is $\AC$ or $\PC$).
Let $D$ be a set of disjoint subsets of $G$, each of which is a union of conjugacy classes of $G$, and such that $c = \bigsqcup_{\gamma \in D} \gamma$.
Let $\xi : D \to \Zpos$ be a map and let $\card{\xi} = \sum_{\gamma \in D} \xi(\gamma)$.

\begin{definition}
	\label{defn:monoid-comp-dxi}
	The \emph{monoid of components $\Comp_X(D, \xi)$} is defined in the following way: its elements of degree $n$ are the elements $x \in \Comp_X(c)$ satisfying
	\[
		\forall \gamma \in D, \quad\quad
		n \xi(\gamma)
		=
		\sum_{\substack{
			\gamma' \in D^*\\
			\gamma' \subseteq \gamma
		}}
			\mu_c(x)(\gamma'),
	\]
	and its multiplication coincides with that of $\Comp_X(c)$.
\end{definition}

In general, $\Comp_X(D, \xi)$ is not exactly a graded submonoid of $\Comp_X(c)$ as the degrees are divided by~$\card{\xi}$.
Elements of degree $n$ of $\Comp_X(D, \xi)$ are in bijection with connected components of $\Hur^*_X(G, D, n\xi)$ (i.e.,
$
	\Comp_X(D, \xi)
	\simeq
	\bigsqcup_{n \geq 0}
		\pi_0 \Hur^*_X(G, D, n \xi)
$
as graded sets).
Note that $\Comp(c) = \Comp(\{c\}, 1)$.
Importantly, we have the following property (cf.~\cite[Corollaire~3.4.18]{SegThese}):

\begin{proposition}
	\label{prop:fingen}
	The graded monoid $\Comp_X(D, \xi)$ is finitely generated.
\end{proposition}

Be aware of the fact that the generators of $\Comp_{\PC}(D, \xi)$ need not all have the same degree, as illustrated in~\cite[Remarques~3.4.16 et 3.4.20]{SegThese}.

\subsection{Chart of notations}
\label{subsn:char-not}

{
	\centering
	\begin{tabular}{|c|c|c|}
		\hline
		\bf Notation &
		\bf Reference &
		\bf Short description
		\\
		\hline
		$D_H, \xi_H, c_H, D^*_H, \tau_H$ & \Cref{defn:dgen} & \\
		$\Omega(D_H)$ & \Cref{defn:splitting-number} & splitting number \\
		$|\gbar|, \pi\gbar, \gen{\gbar}$ & \Cref{ssn:tuples} & size, product, group of a tuple \\
		$\mu_c$ & \Cref{defn:multidisc} & $c$-multidiscriminant \\
		$\tilde\pi$ & \Cref{defn:tilde-pi} & ``abelianized product'' \\
		$\B_n, \sim$ & \Cref{subsubsn:conf-braid,sssn:conn-comp} & braid group, equivalence modulo braids \\
		$\mathrm{(C)}\Hur^*_X(G, D, \xi)$ & \Cref{subsubsn:hurwitz} & Hurwitz spaces of marked $G$-covers \\
		$\Comp_X(c), \Comp_X(D, \xi)$ & \Cref{defn:monoid-comp,defn:monoid-comp-pc,defn:monoid-comp-dxi} & monoids of components of Hurwitz spaces \\
		\hline
	\end{tabular}
}

\paragraph{Notation introduced in later sections.}~
\medskip

{
	\centering
	\begin{tabular}{|c|c|c|}
		\hline
		\bf Notation &
		\bf Reference &
		\bf Short description
		\\
		\hline
		$X, G, D, c, D^*, \xi, \card{\xi}, \tau, \Omega(D)$ & Top of \Cref{sn:exponent} & \it (setup for \Cref{sn:exponent,sn:leading-coeff}) \\
		$\pi_0\CHur^*_X(G, D, n\xi)$ & \Cref{defn:pi0chur} & set of components (defined combinatorially) \\
		$\mathcal{L}_n(G, D, \xi)$ & \Cref{defn:likely-map} & set of likely maps of degree $n$ \\
		$\Pi_c, U(G,c), U_1(G,c)$ & Top of \Cref{ssn:lifting} & lifting invariant (and the groups it lives in)\\
		$S_c, H_2(G,c)$ & \Cref{defn:H2Gc} & groups used to describe $U(G,c)$ and $U_1(G,c)$ \\
		\hline
	\end{tabular}
}

\section
  [Asymptotics of the count of components.
Part~1: the exponent]
  {Asymptotics of the count of components of $\CHur^*_X(G, D, n\xi)$.\break
Part~1: the exponent}
\label{sn:exponent}

The setting for this section is the following:%
\footnote{
	Make sure that you have read \Cref{sssn:organization} before comparing the results as they are stated here to \Cref{thm:asymp-comp-main}.
}
\begin{itemize}
	\item
		$X$ is either the affine or projective complex line;
	\item
		$G$ is a nontrivial finite group;
	\item
		$D$ is a set of disjoint non-empty subsets of~$G$ which collectively generate~$G$, each of which is a union of conjugacy classes of~$G$;
	\item
		$c := \bigsqcup_{\gamma \in D} \gamma$, and $D^*$ is the set of conjugacy classes of $G$ which are contained in $c$.
	\item
		$\xi$ is a map $D \to \Zsp$, and $\card{\xi} := \sum_{\gamma \in D} \xi(\gamma)$.
\end{itemize}
We define the map $\tau : D^* \twoheadrightarrow D$ mapping each $\gamma \in D^*$ to the unique~$\tau(\gamma) \in D$ containing $\gamma$.
The ``intrinsic'' splitting number $\Omega(D)$ is defined as $\card{D^*} - \card{D} = \sum_{\gamma \in D} \left( \card{\tau^{-1}(\gamma)}-1 \right)$ (compare with \Cref{defn:splitting-number}).
Note that $\Omega(D) = 0$ if and only if $D = D^*$, i.e., $D$ consists of conjugacy classes of $G$.

\begin{definition}
	\label{defn:pi0chur}
	We define the set $\pi_0\CHur^*_X(G, D, n\xi)$ of $B_{n\card{\xi}}$-orbits of tuples $\gbar \in G^{n\card{\xi}}$ such that:
	\begin{enumerate}[label=(\roman*)]
		\item
			The elements of $\gbar$ generate $G$, i.e., $\gen{\gbar} = G$.
		\item
			For each $\gamma \in D$, exactly $\xi(\gamma)$ entries of the tuple $\gbar$ belong to $\gamma$.
		\item
			If $X = \PC$, then $\pi\gbar = 1$.
	\end{enumerate}
\end{definition}

The notation $\pi_0\CHur^*_X(G, D, n\xi)$ reflects the fact that this set is in bijection with the set of connected components of the space $\CHur^*_X(G, D, n\xi)$ (cf.~\Cref{sssn:conn-comp}), so that these notations do not conflict for counting purposes.
The goal of this article is to estimate how the cardinality of this set grows as $n \to \infty$.
More precisely, the main result of this section is the following:

\begin{theorem}
	\label{thm:exponent}
	We have the following estimates for $\card{\pi_0\CHur^*_X(G, D, n\xi)}$:
	\begin{thmenum}
		\item
			\label{thmitem:comp-asymp-aff}
			In the affine case, we have:
			\[
				\card{\pi_0\CHur^*_{\AC}(G, D, n\xi)}
				=
				\Theta \bigg( n^{\Omega(D)} \bigg).
			\]
			(Here, the notation $f=\Theta(g)$ means that $f=O(g)$ and $g=O(f)$.)
		\item
			\label{thmitem:comp-asymp-proj}
			In the projective case, we have the upper bound:
			\[
				\card{\pi_0\CHur^*_{\PC}(G, D, n\xi)}
				=
				O\!\left(n^{\Omega(D)}\right),
			\]
			and the lower bound:
			\[
				n^{\Omega(D)}
				=
				\Oo{
					\card{\pi_0\CHur^*_{\PC}(G, D, n\exp(G)\xi)}
				}
			\]
			where $\exp(G)$ denotes the exponent of the group $G$.
	\end{thmenum}
\end{theorem}

Note that \Cref{thmitem:comp-asymp-proj} is only a lower bound on the values taken by the counting functions at multiples of $\exp(G)$.
Therefore, it is possible that some of the several polynomials making up the Hilbert function ($Q_1, \ldots, Q_{W-1}$ in the notation of \Cref{prop:hilbfun-oscpol}, but not $Q_0$) have degree strictly smaller than $\Omega(D)$ (they may even vanish, as in \Cref{rk:symgp}).

\bigskip

The proof of \Cref{thm:exponent} is in four steps:
\begin{itemize}
	\item
		In \Cref{subsn:counting-likely}, we prove that the number of maps $D^* \rightarrow \Zpos$ which are ``likely'' to be $c$-multidiscriminants of components of group $G$ and of degree $n$ (cf.~\Cref{defn:likely-map}) grows like~$n^{\Omega(D)}$ (\Cref{prop:counting-likely-maps}).
	\item
		In \Cref{subsn:lowerbound-exponent}, we prove the lower bounds in \Cref{thm:exponent} (\Cref{prop:lowerbound-exponent}).
		To do so, we associate to every likely map $\psi$ a component of group $G$ whose multidiscriminant is closely related to $\psi$.
	\item
		In \Cref{subsn:factorization}, we prove that one can factor a given component from any component whose multidiscriminant is ``big enough'' in some sense (\Cref{lem:factorization}).
	\item
		In \Cref{subsn:upperbound-exponent}, we prove the upper bounds in \Cref{thm:exponent} (\Cref{prop:upperbound-exponent}).
		The proof uses \Cref{lem:factorization} to show that the number of components with a given multidiscriminant is uniformly bounded (\Cref{lem:bounded-comp-each-multidisc}).
\end{itemize}

\subsection{Counting likely maps}
\label{subsn:counting-likely}

Let $n \in \N$.
If $\gbar \in c^{n\card{\xi}}$ and $\gamma \in D$, then the number of entries of $\gbar$ which belong to $\gamma$ is:
\[
	\sum_{\gamma' \in \tau^{-1}(\gamma)} \mu_c(\gbar)(\gamma')
\]
where $\mu_c(\gbar)$ is the $c$-multidiscriminant from \Cref{defn:multidisc} and the map $\tau : D^* \to D$ is the surjection defined at the beginning of \Cref{sn:exponent} (see also~\Cref{defn:monoid-comp-dxi}).
In particular, the braid orbit of a tuple $\gbar \in c^{n|\xi|}$ (with $\pi\gbar = 1$ if $X = \PC$) belongs to $\pi_0\CHur^*_X(G, D, n\xi)$ if and only if $\gen{\gbar} = G$ and
\[
	\forall \gamma \in D, \quad
	n \xi(\gamma)
	=
	\sum_{\gamma' \in \tau^{-1}(\gamma)}
		\mu_c(\gbar)(\gamma').
\]
We abstract this last condition by defining \emph{likely maps}, i.e., maps $D^* \to \Zpos$ which can reasonably be the $c$-multidiscriminant of an element of $\pi_0\CHur^*_X(G, D, n\xi)$:

\begin{definition}
	\label{defn:likely-map}
	Let $n \in \N$.
	A map $\psi : D^* \to \Zpos$ is a \emph{likely map of degree $n$} if, for all $\gamma \in D$:
	\[
		\sum_{\gamma' \in \tau^{-1}(\gamma)} \psi(\gamma') = n \xi(\gamma).
	\]
	We denote by $\mathcal{L}_n(G, D, \xi)$ the set of all likely maps of degree $n$.
\end{definition}

Then, $\pi_0\CHur^*_X(G, D, n\xi)$ is the set of elements $x \in \Comp_X(c)$ such that $\gen{x} = G$ and $\mu_c(x) \in \mathcal{L}_n(G, D, \xi)$.
In general, a component is not determined by its $c$-multidiscriminant, but we shall see later that there is a uniform bound on the number of components having a given likely map as their $c$-multidiscriminant (\Cref{lem:bounded-comp-each-multidisc}).
This motivates first counting likely maps:


\begin{proposition}
	\label{prop:counting-likely-maps}
	We have the following estimate for the count of likely maps:
	\[
		\card{\mathcal{L}_n(G, D, \xi)}
		\underset{n \to \infty}{=}
		\left( \prod_{\gamma \in D} \frac{\xi(\gamma)^{\card{\tau^{-1}(\gamma)} - 1}}{\left (\card{\tau^{-1}(\gamma)} - 1 \right )!} \right) \, n^{\Omega(D)}
		+
		O\!\left(
			n^{\Omega(D)-1}
		\right).
	\]
\end{proposition}

\begin{proof}
	To determine a likely map of degree $n$, one must choose for each $\gamma \in D$ a way to divide $n \xi(\gamma)$ into the $\card{\tau^{-1}(\gamma)}$ conjugacy classes that make up $\gamma$.
	Therefore, $\card{\mathcal{L}_n(G, D, \xi)}$ is equal to the number of such possibilities, which is given by:
	\[
		\prod_{\gamma \in D} \binom{n \xi(\gamma) + \card{\tau^{-1}(\gamma)} - 1}{\card{\tau^{-1}(\gamma)} - 1}.
	\]
	When $n$ is large enough, this coincides with a polynomial in $n$ of degree:
	\[
		\sum_{\gamma \in D}
			\left(
				\card{\tau^{-1}(\gamma)} - 1
			\right) 
		=
		\card{
			\bigsqcup_{\gamma \in D}
				\tau^{-1}(\gamma)
		} - \card{D}
		= \card{D^*} - \card{D}
		= \Omega(D)
	\]
	and of leading coefficient:
	\[
		\prod_{\gamma \in D}
			\frac
				{
					\xi(\gamma)^{
						\card{\tau^{-1}(\gamma)} - 1
					}
				}
				{
					\left(
						\card{
							\tau^{-1}(\gamma)
						} - 1
					\right)!
				}.
		\qedhere
	\]
\end{proof}

\subsection{Proof of the lower bound on $\card{\pi_0\CHur^*_X(G, D, n\xi)}$}
\label{subsn:lowerbound-exponent}

To obtain the desired lower bound on $\card{\pi_0\CHur^*_X(G, D, n\xi)}$, we describe an injection from the set of likely maps to the set of components, which does not increase the degree ``too much''.

\begin{proposition}
	\label{prop:lowerbound-exponent}
	We have the following lower bounds for $\card{\pi_0\CHur^*_X(G, D, n\xi)}$:
	\begin{propenum}
		\item
			\label{prop:lowerbound-exponent-i}
			If $X=\AC$, then $n^{\Omega(D)} = \Oo{\card{\pi_0\CHur^*_X(G, D, n\xi)}}$.
		\item
			\label{prop:lowerbound-exponent-ii}
			If $X=\PC$, then $n^{\Omega(D)} = \Oo{\card{\pi_0\CHur^*_X(G, D, \exp(G)\, n\, \xi) }}$.
	\end{propenum}
\end{proposition}

\begin{proof}
	First fix a tuple $\underline{h}$ representing an element of $\Comp_X(D, \xi)$ of group $G$, which is possible because $G$ is $D$-generated (cf.~\cite[Proposition~3.2.22]{SegThese}).
	Let $r$ be the degree of that element.
	For each conjugacy class $\gamma \in D^*$, choose an arbitrary element $\tilde \gamma \in \gamma$.
	If $\psi \in \mathcal{L}_n(G, D, \xi)$ is a likely map of degree $n$, define the following tuple, where the concatenation happens in an arbitrary order:
	\[
		\gbar_{\psi}
		:=
		\prod_{\gamma \in D^*} (
			\underbrace{\tilde \gamma, \tilde \gamma, \ldots, \tilde \gamma}_{\psi(c)}
		).
	\]
	We have $\mu_c(\gbar_{\psi}) = \psi$.
	So $\underline{h}  \, \gbar_{\psi}$ is a tuple with group $G$ and $c$-multidiscriminant $\psi + \mu_c(\underline{h}) \in \mathcal{L}_{n+r}(G, D, \xi)$.
	The tuples $\underline{h} \, \gbar_{\psi}$ are pairwise nonequivalent for distinct likely maps $\psi \in \mathcal{L}_n(G, D, \xi)$ as they have distinct $c$-multidiscriminants (cf.~\Cref{sssn:main-inv}).
	This proves the lower bound:
	\[
		\card{\pi_0\CHur^*_{\AC}(G, D, (n+r)\xi)} \geq \card{\mathcal{L}_n(G, D, \xi)}.
	\]
	Coupled with \Cref{prop:counting-likely-maps}, this implies \customref{point (i)}{prop:lowerbound-exponent-i}.

	Now consider the tuples $\left( \underline{h} \gbar_{\psi} \right)^{\exp(G)}$, whose product is one.
	As before, these tuples are nonequivalent for distinct likely maps $\psi \in \mathcal{L}_n(G, D, \xi)$ because they have distinct $c$-multidiscriminants, namely $\exp(G)\left( \psi + \mu_c(\underline{h}) \right)$.
	Thus, these tuples define distinct components belonging to the set $\pi_0\CHur^*_{\PC}\big(G, D,\exp(G)(n+r)\xi\big)$.
	This leads to the lower bound:
	\[
		\card{\pi_0\CHur^*_{\PC}\big(G, D,\exp(G)(n+r)\xi\big)} \geq \card{\mathcal{L}_n(G, D, \xi)},
	\]
	which implies \customref{point (ii)}{prop:lowerbound-exponent-ii} when coupled with \Cref{prop:counting-likely-maps}.
\end{proof} 
 
\subsection{The factorization lemma}
\label{subsn:factorization}

We prove the factorization result \Cref{lem:factorization}, which is a generalized version of \cite[Proposition~3.4]{EVW}.
This lemma comes in handy in minimizing redundancy when counting components.

\begin{lemma}
	\label{lem:factorization}
	Let $\gbar$, $\gbar'$ be tuples of elements of $G$ with $\gen{\gbar'} = G$.
	Assume that, for every conjugacy class $\gamma \subseteq G$ such that the number $n(\gamma)$ of entries of $\gbar$ belonging to $\gamma$ is nonzero, there are at least $\card{\gamma}\ord(\gamma)+n(\gamma)$ entries of $\gbar'$ belonging to $\gamma$.	
	Then, there exists $\gbar''$ such that $\gbar' \sim \gbar \, \gbar''$ and $\gen{\gbar''} = G$.
\end{lemma}

\begin{proof}
	We prove the result by induction on the size of $\gbar$.
	If $\gbar$ is the empty tuple, then $\gbar'' = \gbar'$ always works.
	Consider a tuple $\gbar \in G^n$, and assume that the result is known each time $\gbar$ has size smaller than $n$.
	Let $\gamma$ be the conjugacy class of $g_1$ in $G$.
	By hypothesis, at least $\card{\gamma}\ord(\gamma) + 1$ entries of~$\gbar'$ belong to $\gamma$.
	By the pigeonhole principle, some element $h_1 \in \gamma$ appears at least $\ord(\gamma)+1$ times in~$\gbar'$.
	Apply braids in order to move $\ord(\gamma)+1$ copies of $h_1$ in front:
	\[
		\gbar'
		\sim
		(\underbrace{h_1,\ldots,h_1}_{\ord(\gamma)}, h_1, k_1,\ldots,k_w).
	\]
	The tuple $(h_1,k_1,\ldots,k_w)$ still generates $G$.
	By \Cref{propitem:conjugate-subtuples}, we may use braids to conjugate the block $(\underbrace{h_1,\ldots,h_1}_{\ord(\gamma)})$, whose product is $1$, by an element $\gamma \in G$ such that $h_1^{\gamma} = g_1$.
	Thus:
	\[
		\gbar' \sim (\underbrace{g_1,\ldots,g_1}_{\ord(\gamma)}, h_1, k_1,\ldots,k_w).
	\]
	We have factored $g_1$ from $\gbar'$. Now, we have to factor $(g_2, \ldots, g_r)$ from the tuple:
	\[
		(\underbrace{g_1, \, \ldots, \, g_1}_{\ord(\gamma)-1}, \, h_1, k_1,\ldots,k_w).
	\]
	This tuple satisfies the same hypotheses as $\gbar'$, except that it has one less entry belonging to the class~$\gamma$.
	Since there is also one less element of $\gamma$ to factor, we can apply the induction hypothesis to factor $(g_2, \ldots, g_r)$ from $\gbar'$.
	This concludes the proof.
\end{proof}

Note that if $\gbar$ and $\gbar'$ are tuples of product $1$, then any $\gbar''$ such that $\gbar' \sim \gbar \, \gbar''$ also has product $1$.

Let $\kappa$ be the maximum value of $\card{\gamma} \ord(\gamma)$ over classes $\gamma \in D^*$.
For $\psi \in \Z^{D^*}$ and $\gamma \in D^*$, let:
\[
	\Nn(\psi)(\gamma) =
		\left \lbrace
			\begin{matrix}
				1 & \text{if } \psi(\gamma) \geq 1 \\
				0 & \text{otherwise.}
			\end{matrix}
		\right .
\]
Using this notation, we rephrase a form of \Cref{lem:factorization} in terms of $c$-multidiscriminants:

\begin{corollary}
	\label{cor:factorization-with-multidisc}
	Let $x,y \in \Comp_X(D, \xi)$, with $\gen{y} = G$.
	If $\mu_c(y) \geq \mu_c(x) + \kappa \, \Nn\!(\mu_c(x))$, then there exists a component $z \in \Comp_X(D, \xi)$ such that $y=z x$ and $\gen{z} = G$.
\end{corollary}

\subsection{Proof of the upper bound on $\card{\pi_0\CHur^*_X(G, D, n\xi)}$}
\label{subsn:upperbound-exponent}

In this subsection, we prove \Cref{prop:upperbound-exponent}, which is the upper bound in \Cref{thm:exponent}.
For each map $\psi : D^* \to \Zpos$, we let $F_{\psi}$ be the finite set of elements of $\Comp_{\AC}(c)$ of group $G$ whose $c$-multidiscriminant is $\psi$.
We first prove the following lemma:

\begin{lemma}
	\label{lem:bounded-comp-each-multidisc}
	There is an integer $K$ such that $\card{F_{\psi}} \leq K$ for every map $\psi : D^* \to \Zpos$.
\end{lemma}

\begin{proof}
	Define $\kappa$ as above, and let $K$ be the maximal value of $\card{F_{\psi}}$ over maps $\psi : D^* \to \{0, \ldots, \kappa\}$.
	The number $K$ is a well-defined integer because there are finitely many maps $D^* \to \{0, \ldots, \kappa\}$ and every such map is the $c$-multidiscriminant of finitely many components.

	We prove that $\card{F_{\psi}} \leq K$ for all maps $\psi : D^* \to \Zpos$ by induction on $\max \psi$.
	If $\max \psi \leq \kappa$, then this holds by definition of $K$.
	Now, consider a map $\psi: D^* \to \Zpos$ such that $\max \psi > \kappa$, and assume that $\card{F_{\psi'}} \leq K$ for every map $\psi' : D^* \to \Zpos$ with $\max \psi' < \max \psi$.
	Let $c_1, \ldots, c_r$ be the elements of $D^*$ at which $\psi$ takes its maximal value.
	For each $i \in \{1,\ldots,r\}$, let $g_i$ be an element of $c_i$.
	This defines an $r$-tuple $\gbar = (g_1, \ldots, g_r)$.
	Note that $\mu_c(\gbar)$ takes the value $1$ where $\psi$ is maximal, and takes the value $0$ otherwise.
	Therefore:
	\[
		\psi
		\;\geq\;
		\max(\psi)\mu_c(\gbar)
		\;\geq\;
		(1 + \kappa)\mu_c(\gbar)
		\;=\;
		\mu_c(\gbar)
		+
		\kappa \,
		\Nn\!\!\left(
			\mu_c(\gbar)
		\right).
	\]
	By \Cref{cor:factorization-with-multidisc}, we can factor the component represented by $\gbar$ from any component with group~$G$ and multidiscriminant $\psi$.
	Thus, concatenation with $\gbar$ induces a surjection~$F_{\psi - \mu_c(\gbar)} \twoheadrightarrow F_{\psi}$ and in particular
	$
		\card{F_{\psi}}
		\leq
		\card{F_{\psi - \mu_c(\gbar)}}
	$.
	The maximal value taken by the map $\psi - \mu_c(\gbar)$ is $\max\psi - 1$, so $\card{F_{\psi - \mu_c(\gbar)}} \leq K$ by the induction hypothesis.
\end{proof}

\begin{proposition}
	\label{prop:upperbound-exponent}
	We have $\card{\pi_0\CHur^*_X(G, D, n\xi)} = \Oo{n^{\Omega(D)}}$.
\end{proposition}

\begin{proof}
	Since $\card{\pi_0\CHur^*_{\PC}(G, D, n\xi)} \leq \card{\pi_0\CHur^*_{\AC}(G, D, n\xi)}$, it suffices to prove the result when $X = \AC$.
	We have:
	\begin{align*}
		\card{\pi_0\CHur^*_{\AC}(G, D, n\xi)}
		& = \sum_{\psi \in \mathcal{L}_n(G, D, \xi)} \card{F_{\psi}} \\
		& \leq \card{\mathcal{L}_n(G, D, \xi)} \times K & \text{ by \Cref{lem:bounded-comp-each-multidisc}} \\
		& = O\!\left( n^{\Omega(D)} \right) & \text{by \Cref{prop:counting-likely-maps}}.
		&\qedhere
	\end{align*}
\end{proof}

\section
  [Asymptotics of the count of components.
  Part~2: the leading coefficient]
  {Asymptotics of the count of components of $\CHur^*_X(G, D, n\xi)$.\break
Part~2: the leading coefficient}
	\label{sn:leading-coeff}

The setting for this section is as in \Cref{sn:exponent}: we fix $X$, $G$, $D$, $\xi$ and define $c$, $D^*$, $\tau$, $\Omega(D)$ in the same way.
We obtain additional information on the main asymptotics of the count of connected components of $\CHur^*_X(G, D, n\xi)$ as $n$ grows.
In \Cref{ssn:lifting}, we quickly review definitions and results (notably \Cref{thm:conway-parker}) concerning the lifting invariant.
In \Cref{subsn:most-mbig}, we prove \Cref{lem:few-small}, which implies that the bijection of \Cref{thm:conway-parker} concerns ``most'' components.
In \Cref{subsn:asymp-affine}, we prove the main theorem for $X = \AC$ (\Cref{thm:asymp-affine}, which is \Cref{thmitem:asymp-comp-main-1}).
The case $X = \PC$ is divided as such: \Cref{subsn:really-likely-maps-are-multidisc} contains general results, \Cref{subsn:leadin-coeff-non-splitter} addresses the case $\Omega(D_H)=0$ (\Cref{thmitem:asymp-comp-main-2}), and \Cref{subsn:leadin-coeff-xi-one} focuses on the case $\xi=1$ (\Cref{thmitem:asymp-comp-main-3}).

\subsection{The lifting invariant}
\label{ssn:lifting}

We review a lifting invariant introduced in \cite{EVW2}, based on the invariant from \cite{Fried95}.
A clear and concise presentation, including proofs for the facts stated here, is found in \cite{wood}.

Let $c \subseteq G$ be a union of conjugacy classes of $G$ which collectively generate $G$, and let $D^*$ be the set of all conjugacy classes of $G$ contained in $c$.
Let $U(G, c)$ be the Grothendieck group of $\Comp_{\AC}(c)$, i.e., the group generated by elements $[g]$ for each $g \in c$, satisfying $[g][h] = [h^g][g]$.
There is a monoid homomorphism:
\[
	\Pi_c :
	\left\lbrace
		\begin{matrix}
			\Comp_{\AC}(c) & \to & U(G, c) \\
			(g_1, \ldots, g_n) & \mapsto & [g_1]\cdots[g_n].
		\end{matrix}
	\right .
\]
The \emph{$c$-lifting invariant} of a component $x \in \Comp_{\AC}(c)$ is its image $\Pi_c(x) \in U(G, c)$.
There are group homomorphisms $\pi : U(G, c) \to G$ (induced by $[g] \mapsto g$) and $\mu_c : U(G, c) \to \Z^{D^*}$, recovering the product and $c$-multidiscriminant of a component $x \in \Comp_{\AC}(c)$ from its $c$-lifting invariant:
\[\begin{tikzcd}[ampersand replacement=\&]
	\& {\Comp_{\AC}(c)} \\
	\& {U(G, c)} \\
	G \&\& {\Z^{D^*}} \\
	\& {G^{\ab}}
	\arrow["{\Pi_c}", from=1-2, to=2-2]
	\arrow["\pi"', curve={height=18pt}, two heads, from=1-2, to=3-1]
	\arrow["{\mu_c}", curve={height=-18pt}, from=1-2, to=3-3]
	\arrow["\pi", curve={height=-6pt}, two heads, from=2-2, to=3-1]
	\arrow["{\mu_c}"', curve={height=6pt}, two heads, from=2-2, to=3-3]
	\arrow[two heads, from=3-1, to=4-2]
	\arrow["\tilde\pi", two heads, from=3-3, to=4-2]
\end{tikzcd}.\]
(Here, $\tilde\pi : \Z^{D^*} \to G^{\ab}$ is the group homomorphism defined in \Cref{defn:tilde-pi}.)

We let $U_1(G, c)$ be the kernel of the map $\pi : U(G, c) \to G$, which is a central subgroup of $U(G, c)$ (same proof as \Cref{propitem:product-one-commute}).
The group $U_1(G, c)$ is the Grothendieck group of $\Comp_{\PC}(c)$, and the lifting invariant defines a monoid homomorphism
$
	\Pi_c : \Comp_{\PC}(c) \to U_1(G, c)
$.

\subsubsection{Conway--Parker--Fried--Völklein--type theorems.}

The following theorem, which is \cite[Theorem~3.1]{wood}/\cite[Theorem~7.6.1]{EVW2}, is an improved version of a theorem of Conway and Parker which had already been improved by Fried and Völklein:

\begin{theorem}
	\label{thm:conway-parker}
	There is an integer $M_c$ such that for every map $\psi : D^* \to \Z$ satisfying $\min \psi \geq M_c$, the map $\Pi_c$ induces a bijection between the two following sets:
	\begin{itemize}
		\item
			the set of elements $x \in \Comp_{\AC}(c)$ such that $\mu_c(x) = \psi$ and $\gen{x} = G$;
		\item
			the set of elements $x \in U(G, c)$ such that $\mu_c(x) = \psi$.
	\end{itemize}
\end{theorem}

\Cref{thm:conway-parker} means that $\Pi_c$ is almost a bijection when the multidiscriminants are big enough and covers are connected (i.e., the monoid $\Comp_X(c)$ behaves ``asymptotically'' like a group).
We make this notion of ``bigness'' precise with the following definition:

\begin{definition}
	\label{defn:Mbig}
	Let $M \in \N$.
	A component $x \in \Comp_X(G)$ is \emph{$M$-big} if it is represented by a tuple~$\gbar$ such that every conjugacy class of $\gen{x}$ appearing in $\gbar$ is the conjugacy class of at least $M$ entries of~$\gbar$.
\end{definition}

\begin{remark}
	\label{rk:univ-M}
	Let $M$ be the maximal value of $M_c$ where $c$ is a union of conjugacy classes of~$G$ which collectively generate~$G$.
	Then, \Cref{thm:conway-parker} says that an $M$-big component $x$ with $\gen{x} = G$ is entirely determined by its $G$-lifting invariant.
\end{remark}

\subsubsection{Description of $U(G, c)$ and $U_1(G, c)$.}
\label{subsubsn:desc-ugc}

Fix a Schur extension $S \twoheadrightarrow G$, i.e., a central extension~$S$ of $G$ by $H_2(G,\Z)$ such that $H_2(G,\Z) \subseteq [S,S]$.
By definition, it fits in an exact sequence:
\[
	1
	\to H_2(G, \Z)
	\to S
	\overset{p}{\to} G
	\to 1.
\]
Let $Q_c$ be the subgroup of $H_2(G, \Z)$ generated by commutators $[\hat x, \hat y]$ between elements of $S$ whose projections in $G$ are commuting elements of $c$:
\[
	Q_c
	:=
	\sbox0{$
		[\hat x, \hat y]
		\,\,\vast| \,\,
		\begin{matrix}
			\hat x, \hat y \in S \\
			p(\hat x), p(\hat y) \in c \\
			p(\hat x)p(\hat y) = p(\hat y)p(\hat x)
		\end{matrix}
	$}
	\mathopen{\resizebox{1.2\width}{\ht0}{$\Bigg\langle$}}
	\usebox{0}
	\mathclose{\resizebox{1.2\width}{\ht0}{$\Bigg\rangle$}}
	.
\]

\begin{definition}
	\label{defn:H2Gc}
	We define the quotient groups $S_c := S/Q_c$ and $H_2(G, c) := H_2(G, \Z)/Q_c$.
\end{definition}

\begin{theorem}
	\label{thm:desc-ugc}
	Both $U(G, c)$ and $U_1(G, c)$ admit descriptions as (fiber) products:
	\begin{enumerate}[label=(\roman*)]
		\item
			\label{thm:desc-ugc-i}
			$
				U(G,c)
				\simeq
				S_c
				\underset{G^{\ab}}{\times}
				\Z^{D^*}
			$
		\item
			\label{thm:desc-ugc-ii}
			$
				U_1(G,c)
				\simeq
				H_2(G, c)
				\times
				\ker \tilde\pi
			$,
			where $\tilde\pi : \Z^{D^*} \to G^{\ab}$ is as in \Cref{defn:tilde-pi}.
	\end{enumerate}
\end{theorem}

\begin{proof}
	Point~\ref{thm:desc-ugc-i} is \cite[Theorem~2.5]{wood}.
	The second point follows:
	\[
		U_1(G,c) = \ker(U(G,c) \to G) \simeq \ker(S_c \to G) \times \ker(\Z^{D^*} \to G^{\ab}) = H_2(G,c) \times \ker \tilde\pi.
		\qedhere
	\]
\end{proof}

\subsection{Most components are $M$-big}
\label{subsn:most-mbig}

Let $M$ be a constant as in \Cref{rk:univ-M}, i.e., such that $M$-big components (\Cref{defn:Mbig}) of group~$G$ are determined by their $c$-lifting invariant.

\begin{lemma}
	\label{lem:few-small}
	For $n \in \N$, denote by $F^{M\text{-small}}_X(G, D, n\xi)$ the set of components $x \in \Comp_X(D, \xi)$ of group $G$ and degree $n$ such that $\min \mu_c(x) < M$.
	Then:
	\[
		\card{F^{M\text{-small}}_X(G, D, n\xi)} = \Oo{n^{\Omega(D)-1}}.
	\]
\end{lemma}

This should be compared with \Cref{prop:counting-likely-maps}: when $n$ grows, $100\%$ of components are $M$-big. 

\begin{proof}
	Let us count likely maps $\psi$ of degree $n$ such that $\min \psi \geq M$.
	To fix such a likely map, for each $\gamma \in D$, we first put aside $M$ occurrences for each of the $\card{\tau^{-1}(\gamma)}$ conjugacy class making up~$\gamma$.
	Then, for each $\gamma \in D$, we have to choose how to split the $n \xi(\gamma) - M\card{\tau^{-1}(\gamma)}$ remaining occurrences between the $\card{\tau^{-1}(\gamma)}$ conjugacy classes that $\gamma$ consists of.
	The number of ways to do so is:
	\[
		\prod_{\gamma \in D}
		\binom{
			\bigg( n \xi(\gamma) - M \card{\tau^{-1}(\gamma)} \bigg) + \card{\tau^{-1}(\gamma)} - 1
		}{
			\card{\tau^{-1}(\gamma)} - 1
		},
	\]
	which coincides, for $n$ large, with a polynomial having same leading monomial as~$\card{\mathcal{L}_n(G, D, \xi)}$ (cf.~\Cref{prop:counting-likely-maps}).
	Hence, the number of likely maps $\psi \in \mathcal{L}_n(G, D, \xi)$ satisfying $\min \psi < M$ is a $\Oo{n^{\Omega(D)-1}}$.
	Using \Cref{lem:bounded-comp-each-multidisc}, it follows that
	$
		\card{F^{M\text{-small}}_X(G, D, n\xi)} = \Oo{n^{\Omega(D)-1}}
	$.
\end{proof}

\subsection{The case where $X$ is the affine line}
\label{subsn:asymp-affine}

Denote by $[G, G]$ the commutator subgroup of $G$, so that $\card{G} = \card{G^{\ab}} \cdot \card{[G, G]}$.

\begin{proposition}
	\label{prop:compon-per-multidisc-affine}
	Let $\psi \in \Z^{D^*}$ be a likely map such that $\min \psi \geq M$.
	Then, there are exactly $\card{[G,G]} \cdot \card{H_2(G, c)}$ elements of $\Comp_{\AC}(D, \xi)$ of group $G$ whose $c$-multidiscriminant is $\psi$.
\end{proposition}

\begin{proof}
	By \Cref{thm:conway-parker}, components of group $G$ and multidiscriminant $\psi$ are in bijection with elements of $U(G, c)$ whose image in $\Z^{D^*}$ is $\psi$.
	By \iref{thm:desc-ugc}{thm:desc-ugc-i}, we have an isomorphism $U(G, c) \simeq S_c \underset{G^{\ab}}{\times} \Z^{D^*}$.
	Therefore, elements of $U(G, c)$ whose image in $\Z^{D^*}$ is $\psi$ are in bijection with elements of $S_c$ whose image in $G^{\ab}$ is $\tilde\pi(\psi)$.
	Recall that $\ker(S_c \to G) = H_2(G,c)$, so the kernel of the surjection $S_c \to G^{\ab}$ has cardinality $\card{[G, G]} \cdot \card{H_2(G, c)}$, and thus $\tilde\pi(\psi)$ is the image of exactly $\card{[G, G]} \cdot \card{H_2(G, c)}$ elements of $S_c$.
	We conclude that the number of components of group $G$ and multidiscriminant $\psi$ is
	$
		\card{[G, G]} \cdot \card{H_2(G, c)}
	$.
\end{proof}


\begin{theorem}
	\label{thm:asymp-affine}
	We have the following asymptotic estimates:
	\[
		\card{\pi_0\CHur^*_{\AC}(G, D, n\xi)}
		\underset{n\to\infty}{=}
		\left(
			\card{[G, G]}
			\cdot
			\card{H_2(G, c)}
			\cdot
			\prod_{\gamma \in D}
				\frac{
					\xi(\gamma)^{
						\card{
							\tau^{-1}(\gamma)
						} - 1
					}
				}{
					\left(
						\card{
							\tau^{-1}(\gamma)
						} - 1
					\right)!
				}
		\right)
		n^{\Omega(D)}
		+
		\Oo{n^{\Omega(D)-1}}.
	\]
\end{theorem}

\begin{proof}
	We have:
	\begin{align*}
		\card{\pi_0\CHur^*_{\AC}(G, D, n\xi)}
		& =
		\sum_{\psi \in \mathcal{L}_n(G, D, \xi)}
		\card{F_{\psi}} \\
		& =
		\left(
			\sum_{\min(\psi) > M}
				\card{F_{\psi}}
		\right)
		+
		F^{M\text{-small}}_{\AC}(G, D, n\xi).
	\end{align*}
	Apply \Cref{prop:compon-per-multidisc-affine} and \Cref{lem:few-small} to obtain:
	\[
		\card{\pi_0\CHur^*_{\AC}(G, D, n\xi)}
		=
		\card{[G, G]}
		\cdot
		\card{H_2(G, c)}
		\cdot
		\card{\mathcal{L}_n(G, D, \xi)}
		+
		\Oo{n^{\Omega(D)-1}}.
	\]
	We conclude using \Cref{prop:counting-likely-maps}.
\end{proof}

\subsection{Really-likely maps}
\label{subsn:really-likely-maps-are-multidisc}

For the remainder of this section, we focus exclusively on the case $X = \PC$.
First, we observe an obstruction for a given likely map to be the $c$-multidiscriminant of a component of $\CHur^*_{\PC}(G, D, n\xi)$.
Recall from \Cref{defn:tilde-pi} that there is a group homomorphism $\tilde\pi : \Z^{D^*} \to G^{\ab}$ such that $\tilde\pi(\mu_c(x))$ is the image in $G^{\ab}$ of $\pi(x)$ for all $x \in \Comp_{\AC}(D, \xi)$.
Hence, if $x \in \Comp_{\PC}(D, \xi) = \ker \pi$, it is necessary that $\mu_c(x) \in \ker \tilde\pi$.
This motivates the following definition:
\begin{definition}
	\label{defn:really-likely-map}
	A \emph{really-likely map} is a likely map $\psi \in \mathcal{L}(G, D, \xi)$ such that $\tilde \pi(\psi)=1$.
\end{definition}

Fix $M$ as in \Cref{thm:conway-parker}.
We show that the obstruction above is the only obstruction when $\min\psi \geq M$, and we determine the exact number of components whose $c$-multidiscriminant is~$\psi$:

\begin{proposition}
	\label{prop:counting-components-really-likely}
	Let $\psi \in \Z^{D^*}$ be a really-likely map satisfying $\min\psi \geq M$.
	Then, $\psi$ is the $c$-multidiscriminant of exactly $\card{H_2(G, c)}$ elements of $\Comp_{\PC}(D, \xi)$ of group $G$.
\end{proposition}

\begin{proof}
	By \Cref{thm:conway-parker}, elements of $\Comp_{\PC}(D, \xi)$ of group $G$ with multidiscriminant $\psi$ are in bijection with elements of $U_1(G, c)$ whose image in $\Z^{D^*}$ is $\psi$.
	By \iref{thm:desc-ugc}{thm:desc-ugc-ii}, we have an isomorphism $U_1(G, c) \simeq H_2(G, c) \times \ker \tilde\pi$.
	Here, $\ker \tilde\pi$ is the subgroup of $\Z^{D^*}$ consisting of elements whose image in $G^{\ab}$ is $1$, which is the case of really-likely maps.
	So, elements of $U_1(G, c)$ whose image in $\Z^{D^*}$ is $\psi$ are in bijection with elements of $H_2(G, c)$.
	This concludes the proof.
\end{proof}

Thus, the main remaining task is to count really-likely maps of degree $n$.

\subsection{Situation $1$: The non-splitting case}
\label{subsn:leadin-coeff-non-splitter}

We first focus on the case $\Omega(D) = 0$.
In this case, $D$ is a set of conjugacy classes of $G$, i.e., $D = D^*$, and $\tau$ is the identity.
For each $n$, there is exactly one likely map of degree~$n$, namely~$n\xi$.
Whether it is a really-likely map only depends on the order $k := \ord \, \tilde\pi(\xi)$ of the element~$\tilde \pi(\xi)$ in~$G^{\ab}$, which can be computed explicitly in concrete situations.
For every $n$, either~$n$ is not a multiple of~$k$ and there is no component of degree $n$, or $n$ is a multiple of~$k$ and there is exactly one really-likely map of degree $n$.
When $n \geq M$, this map only takes values $\geq M$, so \Cref{prop:counting-components-really-likely} gives the exact count of components of $\CHur^*_{\PC}(G, D, n\xi)$ for large enough $n$:

\begin{proposition}
	\label{prop:hilb-func-non-splitt}
	For $n \in \N$, we have:
	\[
		\card{\pi_0\CHur^*_{\PC}(G, D, n\xi)}
		=
		\left \lbrace
		\begin{matrix}
			0 & \text{if } n \not\in k\N \\
			\card{H_2(G, c)} & \text{otherwise, provided } n \geq M
		\end{matrix}
		\right . .
	\]
\end{proposition}

In particular, an average order of $\card{\pi_0\CHur^*_{\PC}(G, D, n\xi)}$ is given by the constant $\frac{1}{k} \card{H_2(G, c)}$.

\subsection{Situation $2$: When $\card{\xi}=1$}
\label{subsn:leadin-coeff-xi-one}

We now discuss the case where $D=\{c\}$ is a singleton and $\xi(c) = 1$, i.e., we are counting elements of $\Comp_{\PC}(c)$ of group $G$ and degree $n$.
Let $c_1, \ldots, c_s$ be the elements of $D^*$, with $s = \Omega(D) + 1$.
In this situation, multidiscriminants and likely maps are seen as (special) $s$-tuples of elements of $\Z$.
Let $\tilde c_1, \ldots, \tilde c_s$ be the projections in $G^{\ab}$ of the elements of $c_1, \ldots, c_s$ respectively.
If~$\psi \in \Z^s$, define:
\[
	\tilde \pi(\psi) := \prod_{i=1}^s {\tilde c_i}\,\!^{\psi(i)}
	\in G^\ab.
\]

\begin{lemma}
	\label{lem:tildepi-is-surjective}
	The group homomorphism $\tilde \pi : \Z^s \rightarrow G^{\ab}$ is surjective.
\end{lemma}

\begin{proof}
	Consider an element $\tilde g \in G^{\ab}$ and fix an arbitrary lift $g \in G$ of $\tilde g$.
	Since $c$ generates $G$, $g$ factors as a product $g_1 \cdots g_r$ of elements of $c$.
	Let $\psi \in \Z^s$ be the map that counts for each $i=1,\ldots,s$ the number of elements of $c_i$ in this factorization.
	In $G^{\ab}$, we have $\tilde g = \tilde \pi(\psi)$.
\end{proof}

\begin{proposition}
	\label{prop:counting-really-likely-maps}
	As $n \to \infty$, the number of really-likely maps of degree $\leq n$ is
	\[
		\frac
		{n^s}
		{\card{G^{\ab}}s!}
		+
		O\!\left(
			n^{s-1}
		\right)
		.
	\]
\end{proposition}

\begin{proof}
	Let $\Delta = \exp(G^{\ab})$.
	For any $\psi \in \Z^s$, the value of $\tilde \pi (\psi)$ depends only on the class of $\psi$ in~$(\Z/\Delta\Z)^s$.
	By \Cref{lem:tildepi-is-surjective}, $\tilde \pi$ is surjective, so the induced group homomorphism $\tilde \pi : (\Z/\Delta\Z)^s \rightarrow G^{\ab}$ is surjective too and its kernel has size
	$
		\frac{\Delta^s}{\card{G^{\ab}}}
	$.

	Let $X_n$ be the subset of $(\Zpos)^s$ defined by the condition $x_1 + x_2 + \cdots + x_s \leq n$, i.e., the set of likely maps of degree $\leq n$.
	Consider some $\underline{a}=(a_1,\ldots,a_s) \in (\Z/\Delta\Z)^s$.
	For each $i = 1,\ldots,s$, denote by~$b_i$ the element of $\{0,\ldots,\Delta-1\}$ whose class modulo $\Delta$ is $a_i$.
	Let $S(\underline{a}) = \sum_{i=1}^s b_i$.
	If $k_1, \ldots, k_s$ are nonnegative integers and $x_i = b_i + k_i \Delta$, then we have:
	\[
		\sum_{i=1}^s
			x_i
		=
		\sum_{i=1}^s
			\Big(
				b_i + k_i \Delta
			\Big)
		=
		\left (
			\sum_{i=1}^s b_i
		\right )
		+
		\left (
			\sum_{i=1}^s k_i
		\right )
		\Delta
		=
		S(\underline{a})
		+
		\left (
			\sum_{i=1}^s k_i
		\right )
		\Delta.
	\]
	Thus, there are as many elements of $X_n$ whose projection in $(\Z/\Delta\Z)^s$ is $\underline{a}$ as there are tuples $(k_1,\ldots,k_s) \in (\Zpos)^s$ whose sum is at most $\left\lfloor \frac{n -S(\underline{a}) }{\Delta} \right\rfloor$, i.e.:
	\[
		\binom
			{
				s+
				\left\lfloor
					\frac
						{n -S(\underline{a}) }
						{\Delta}
				\right\rfloor
			}
			{s}.
	\]
	For large enough $n$, this expression coincides with a polynomial in $n$ of leading term $ \displaystyle \frac{1}{s!} \left ( \frac{n}{\Delta} \right )^{s}$.
	Putting everything together, the number of really-likely maps of degree $\leq n$, i.e., the number of elements $\psi \in X_n$ such that $\tilde \pi(\psi)=1$, is asymptotically equal, as $n \to \infty$, to:
	\[
		\frac
			{{\Delta}^s}
			{\card{G^{\ab}}}
		\times
		\left(
			\frac{1}{s!}
			\left (
				\frac{n}{\Delta}
			\right )^s
			+
			O\!\left(
				n^{s-1}
			\right)
		\right)
		=
		\frac
		{n^s}
		{\card{G^{\ab}}s!}
		+
		O\!\left(
			n^{s-1}
		\right).
		\qedhere
	\]
\end{proof}

\begin{corollary}
	\label{cor:leading-coeff}
	We have:
	\[
		\sum_{k=0}^n
			\card{
				\pi_0\CHur^*_{\PC}(G, \{c\},k)
			}
		\underset{n\to\infty}{=}
		\frac
			{\card{H_2(G, c)} n^s}
			{\card{G^{\ab}}s!}
		+
		O\!\left(
			n^{s-1}
		\right).
	\]
	In particular, an average order of $\card{\pi_0\CHur^*_{\PC}(G, \{c\}, n)}$ is given by:
	\[
		\frac{
			\card{H_2(G, c)} n^{s-1}
		}{
			\card{G^{\ab}}(s-1)!
		}
	\]
\end{corollary}

\begin{proof}
	Start by separating components according to their multidiscriminants:
	\[
		\sum_{k=0}^n \card{\pi_0\CHur^*_{\PC}(G, \{c\},k)}
		=
		\sum_{
			\substack{
				\psi \text{ really-likely} \\
				\text{of degree } \leq n \\
				\min(\psi) \geq M
			}
		}
			\card{F_{\psi}}
		+
		\sum_{
			\substack{
				\psi \text{ really-likely} \\
				\text{of degree } \leq n \\
				\min(\psi) < M
			}
		}
		\card{F_{\psi}}.
	\]
	Now, apply \Cref{prop:counting-components-really-likely} and \Cref{lem:few-small} to obtain:
	\[
		\sum_{k=0}^n \card{\pi_0\CHur^*_{\PC}(G, \{c\},k)}
		=
		\sum_{
			\substack{
				\psi \text{ really-likely} \\
				\text{of degree } \leq n
			}
		}
		\card{H_2(G, c)} + \Oo{n^{s-1}}.
	\]
	Finally, using \Cref{prop:counting-really-likely-maps}:
	\begin{align*}
		\sum_{k=0}^n
			\card{
				\pi_0\CHur^*_{\PC}(G, \{c\},k)
			}
		& =
		\left (
			\frac
				{n^s}
				{\card{G^{\ab}}s!}
			+
			O\!\left(
				n^{s-1}
			\right)
		\right )
		\card{H_2(G, c)}
		+ \Oo{n^{s-1}} \\
		& =
		\frac{
			\card{H_2(G, c)}
			n^s
		}{
			\card{G^{\ab}}
			s!
		}
		+
		O\!\left(
			n^{s-1}
		\right).
		\qedhere
	\end{align*}
\end{proof}

\begin{remark}
	\label{rk:average-leadcoeff}
	Let $W \in \N$ and $Q_0, \ldots, Q_{W-1} \in \Q[X]$ be as in \Cref{prop:hilbfun-oscpol}, so that:
	\begin{equation}
		\label{eqn:periodic-polynomials}
		\card{\pi_0\CHur^*_{\PC}(G, \{c\}, n)} = Q_{n \bmod W}(n)
		\quad\quad
		\textnormal{for large enough } n.
	\end{equation}
	By \Cref{thmitem:comp-asymp-proj}, we have $\deg Q_i \leq s-1$ for all $i$, and there is an $i$ such that $\deg Q_i = s-1$.
	By \Cref{prop:hilbfun-oscpol}, we have $\deg Q_0 \geq \deg Q_i$ for all $i$, and thus $\deg Q_0 = s-1$.
	Let $q_i$ be the coefficient in front of $n^{s-1}$ in $Q_i$.
	This coefficient may vanish, but it is always nonnegative, and $q_0$ is nonzero.
	\Cref{eqn:periodic-polynomials} implies:
	\[
		\sum_{i=0}^n \card{\pi_0\CHur^*_{\PC}(G, \{c\},i)}
		\underset{n\to\infty}{\sim}
		\left (
			\frac{1}{W} \sum_{m=0}^{W-1} q_m
		\right )
		\frac{n^s}{s}.
	\]
	Together with \Cref{cor:leading-coeff}, this lets us compute the ``average leading coefficient'':
	\[
		\frac{1}{W} \sum_{m=0}^{W-1} q_m
		=
		\frac{
			\card{H_2(G, c)}
		}{
			\card{G^{\ab}}(s-1)!
		}.
	\]
	A noteworthy particular case is obtained when $\Comp_{\PC}(c)$ is generated by components with the same degree~$d$.
	We can then take $W=d$, and we have $q_0 > 0$, $q_1=q_2=\ldots=q_{d-1}=0$.
	Then:
	\[
		q_0
		=
		\frac{
			d\card{H_2(G, c)}
		}{
			\card{G^{\ab}}(s-1)!
		}.
	\]
	In this special case, $\card{\pi_0\CHur^*_{\PC}(G, \{c\}, n)}$ is accurately described as follows:
	\[
		\left \lbrace
			\begin{matrix}
				\card{\pi_0\CHur^*_{\PC}(G, \{c\}, n)} & = & 0 \quad\quad \text{if } n \not\in d\N \\ \\
				\card{\pi_0\CHur^*_{\PC}(G, \{c\},dn)} & \underset{n \to\infty}{\sim} & \displaystyle \frac{d^s\card{H_2(G, c)}}{\card{G^{\ab}}(s-1)!} n^{s-1}.
			\end{matrix}
		\right .
	\]
\end{remark}

\clearpage
\printbibliography
\end{document}